%% file: arxiv.tex
\documentclass[11pt]{article}

\usepackage[margin=1in]{geometry}
\usepackage{times}
\usepackage{microtype}
\usepackage{graphicx}
\usepackage{xcolor}
\usepackage{url}
\usepackage[authoryear,round]{natbib}
\setcitestyle{authoryear,round,citesep={;},aysep={,},yysep={;}}
\usepackage{threeparttable,multicol,multirow}
\usepackage{makecell}
\usepackage{enumerate,enumitem}
\usepackage{algpseudocode,algorithm}
\usepackage{mathtools}
\usepackage{amsthm,listings,booktabs}
\usepackage{authblk}
\usepackage{hyperref}
\hypersetup{
    colorlinks=true,
    linkcolor=blue,
    citecolor=blue,
    urlcolor=blue
}

\input{math.tex}

\title{Optimal Convex Optimization with Inexact Second-Order Oracles}

\author[1]{Lesi Chen}
\author[2]{Chengchang Liu\thanks{Corresponding author: \texttt{liuchengchang@westlake.edu.cn}.}}
\author[3]{Luo Luo}
\author[4]{John C.S. Lui}
\author[1]{Jingzhao Zhang}

\affil[1]{IIIS, Tsinghua University}
\affil[2]{Department of Artificial Intelligence, Westlake University}
\affil[3]{School of Data Science, Fudan University}
\affil[4]{Department of Computer Science and Engineering, The Chinese University of Hong Kong}

\date{}

\begin{document}

\maketitle

\begin{abstract}
In this paper, we 
present a novel second-order method called Accelerated Inexact Newton Extragradient (AINE) for convex optimization using $\delta$-inexact Hessians. We show that AINE can find an $\epsilon$-solution in the inexact second-order oracle (ISO) complexity of $\gO( (\delta/\epsilon)^{1/2} + (L_2/\epsilon)^{2/7} )$ when the Hessian is $L_2$-Lipschitz continuous, and a better complexity of $\gO( (\delta/\epsilon)^{1/2} + (L_3/\epsilon)^{1/5} )$ when the third-order derivative is $L_3$-Lipschitz continuous. Notably, each iteration of our method can be conducted in the same running time as matrix multiplication up to logarithmic factors. In addition, we also establish matching oracle complexity lower bounds for both setups, demonstrating the optimality of our methods.
\end{abstract}

\section{Introduction}
We consider the following optimization problem:
\begin{align} \label{prob:main}
    \min_{\vx \in \sR^d} f(\vx) ,
\end{align}
where the objective $f(\vx)$ is continuously differentiable.
First-order methods for convex optimization have been extensively studied, and complete theories for them have been established. For the task of finding a $\epsilon$-solution, \textit{i.e.} a point $\hat \vx \in \sR^d$ that satisfies $f(\hat \vx) - \min_{\vx \in \sR^d} f(\vx) \le \epsilon$,
\citet{nemirovskij1983problem} established an $\Omega((L_1/\epsilon)^{1/2})$ lower bound on the first-order oracle (FO) complexity for optimizing functions with $L_1$-Lipschitz continuous gradients.
Shortly after, \citet{nesterov1983method} proposed the 
accelerated gradient descent (AGD) method and showed that it achieved the optimal FO complexity of $\gO( (L_1/\epsilon)^{1/2} )$.


Over the past two decades, there has been a growing interest in studying globally convergent second-order methods. In the seminal work, \citet{nesterov2006cubic} introduced the cubic regularized Newton (CRN) method and showed that it achieves the second-order oracle (SO) complexity of $\gO((L_2/\epsilon)^{1/2})$ when the objective has $L_2$-Lipschitz continuous Hessians.
Later,
\citet{nesterov2008accelerating} proposed an accelerated CRN method to achieve a better SO complexity of $\gO((L_2/\epsilon)^{1/3})$. \citet{monteiro2013accelerated} proposed a novel accelerated Newton proximal extragradient (A-NPE) method that can achieve the SO complexity of $\gO( (L_2/\epsilon)^{2/7} \log (L_2/\epsilon))$.
When the objective has $L_3$-Lipschitz continuous third-order derivatives,  
\citet{nesterov2021implementable} proposed a superfast second-order method that achieves the SO complexity of $\gO( (L_3/\epsilon)^{1/4} \log (L_3/\epsilon))$.
Then, three independent groups \citep{gasnikov2019optimal,bubeck2019near,jiang2021optimal} obtained a lower SO complexity of $\gO( (L_3/\epsilon)^{1/5} \log (L_3/\epsilon))$ by extending the A-NPE method. Very recently, two independent groups \citep{kovalev2022first,carmon2022optimal} proposed novel techniques to reduce the logarithmic factors, achieving the optimal SO complexity of $\gO( \min\{ (L_2/\epsilon)^{2/7}, \allowbreak (L_3/\epsilon)^{1/5} \} )$ that matches the oracle lower bounds for both settings \citep{arjevani2019oracle}.

Although optimal SO complexities have been established for second-order methods, implementing these methods with exact Hessian oracles could be expensive. Therefore, one may use  Hessian approximations rather than the exact ones. \citet{ghadimi2017second} proposed the inexact Newton with cubic regularization (INCR) and its acceleration (AINCR) that achieves the inexact second-order oracle (ISO) complexity of $\gO(\delta/\epsilon + (L_2/\epsilon)^{1/2})$ and $\gO((\delta/\epsilon)^{1/2} + (L_2/\epsilon)^{1/3})$, respectively. Subsequently, \citet{agafonov2024inexact,agafonovadvancing} proposed the accelerated inexact tensor method (AITM) that further requires a $\delta'$-inexact third-order derivative information and
achieves the inexact third-order oracle complexity of $\gO( (\delta/\epsilon)^{1/2} + (\delta'/\epsilon)^{1/3} + (L_3/\epsilon)^{1/4})$. However, we observe several limitations in these works: First, these inexact methods do not achieve optimal complexity in the exact setting when $\delta = 0$ \citep{arjevani2019oracle}. Second, the AITM method \citep{agafonov2024inexact,agafonovadvancing} requires querying $\delta'$-inexact third-order derivatives, and its convergence rate depends on an extra parameter $\delta'$. 



In this paper, we propose a novel inexact second-order methods called  Accelerated Inexact Newton Extragradient (AINE) that generalizes the optimal exact methods \citep{carmon2022optimal,kovalev2022first} and addresses both issues above. We show that our method achieves an improved ISO complexity of $\gO( (\delta/\epsilon)^{1/2} + (L_2/\epsilon)^{2/7} )$ for functions with $L_2$-Lipschitz continuous Hessians, and that of $\gO( (\delta/\epsilon)^{1/2} + (L_3/\epsilon)^{1/5} )$ for functions with $L_3$-Lipschitz continuous third-order derivatives. Moreover, we extend the techniques by \citet{nesterov2021implementable,nesterov2021inexact,nesterov2021superfast} to propose an efficient second-order subroutine that implements the required third-order tensor step in the latter case in almost matrix multiplication time, which also removes the $\delta'$ dependency since the inexact third-order derivatives are not used at all. Finally, we also provide matching oracle lower bounds to demonstrate the optimality of our methods.  We summarize our results in Table \ref{tab:main-res} and introduce additional related works in Appendix \ref{sec:related}.

\begin{table}[t]
    \centering
    \caption{Comparison of the inexact second-order oracle (ISO) complexity and per iteration costs of second-order methods with $\delta$-inexact Hessians to find an $\epsilon$-solution of for a convex function with $L_2$-Lipschitz Hessians \textbf{(top)} and with $L_3$-Lipschitz continuous third-order derivatives \textbf{(bottom)}.}
    \label{tab:main-res}
    \begin{threeparttable}
    \begin{tabular}{c c c}
    \hline
    Method &  ISO Complexity  & Per Iter. Cost {\color{blue} \textsuperscript{(a)}} \\ 
    \hline \addlinespace
    \makecell{INCR \\\citep{ghadimi2017second}}  & 
    $\gO\left( \delta/ \epsilon + (L_2/ \epsilon)^{1/2} \right)$ 
   & $\gO(d^\omega + C d)$  \\ \addlinespace
     \makecell{AINCR\\\citep{ghadimi2017second}}  & $\gO\left( (\delta/ \epsilon)^{1/2} + (L_2/ \epsilon)^{1/3} \right)$ & 
    $\gO(d^\omega + C d)$ \\ \addlinespace
   \makecell{AINE \\(This Paper)}  &
   $\gO\left( (\delta/ \epsilon)^{1/2} + (L_2/ \epsilon)^{2/7} \right)$
   & $\gO(d^\omega + C d)$ \\  \addlinespace
 \makecell{Lower Bound \\ \citep{agafonovadvancing}} &
$\Omega\left( (\delta/ \epsilon)^{1/2} + (L_2/ \epsilon)^{2/7} \right)$
& - \\ 
\hline \addlinespace
\makecell{AITM \\ \citep{agafonov2024inexact,agafonovadvancing}} & 
$\gO\left( (\delta/ \epsilon)^{1/2} + (\delta'/\epsilon)^{1/3} + (L_3/ \epsilon)^{1/4} \right)$
   {\color{blue} \textsuperscript{(b)}} &  $\tilde \gO(d^\omega + C d)$ {\color{blue} \textsuperscript{(c)}} \\ \addlinespace
\makecell{AINE \\(This Paper)} & $\gO\left( (\delta/ \epsilon)^{1/2} + (L_3/ \epsilon)^{1/5} \right)$ & $\tilde \gO(d^\omega + C d)$ 
   \\ \addlinespace
\makecell{Lower Bound \\(This Paper)} & $\Omega\left( (\delta/ \epsilon)^{1/2} + (L_3/ \epsilon)^{1/5} \right)$ & - \\ 
\hline
\end{tabular}
    \end{threeparttable}
    \begin{tablenotes}
\footnotesize
\item {\color{blue}(a)} We let $d^\omega \approx d^{2.37}$ to denote the running time of multiplying two $d \times d$ matrices~\citep{duan2023faster,alman2025more}, and we follow the convention \citep{nesterov2006cubic,nesterov2021implementable,nesterov2021inexact,nesterov2021superfast} to assume the cost of solving the one-dimensional convex sub-problem is a constant $C$.
\item {\color{blue}(b)} \citet{agafonov2024inexact} requires querying $\delta'$-inexact third-order oracle.
\item {\color{blue}(c)} \citet{agafonov2024inexact} did not provide a rigorous analysis of the per iteration costs, but it can be obtained as we discuss in Remark \ref{rmk:inexact-AITM-aux}.
\end{tablenotes}
\end{table}

\section{Preliminaries} \label{sec:pre-2nd}

\paragraph{Notations.} Throughout this paper, we use $\Vert \cdot \Vert$ to denote both the Euclidean norm of a vector and the operator norm of a matrix/tensor. 
We let $\sB_D^d(\vc)$ denote the ball $ \{  \vx \in \sR^d : \Vert \vx - \vc \Vert \le D \}$.
For any set $S$ and functions $g, h :  S \rightarrow [0, \infty)$ we write $g = \gO(h)$ or $h = \Omega(g)$ equivalently if
there exists $c > 0$ such that $g(s) \le c  h(s)$ for every $s \in S$. We also write 
$h = \Theta(g)$ if $h = \gO(g) $ and $  h = \Omega(g)$, and write $\tilde \gO(\,\cdot\,)$ to hide the logarithmic factors in $\gO(\,\cdot\,)$.

\subsection{Assumptions}

Throughout this paper, we assume that the objective function $f: \sR^d  \rightarrow \sR$ is convex
and has Lipschitz continuous $p$th-order derivatives.

\begin{asm} \label{asm:convex}
We assume that $f:\sR^d \rightarrow \sR$ is convex and has at least one solution $\vx^* \in \arg \min_{\vx \in \sR^d} f(\vx)$.
\end{asm}

\begin{asm} \label{asm:lip-pth-deriv}
We assume that the function $f: \sR^d \rightarrow \sR$ is $p$ times continuous differentiable and has $L_p$-Lipschitz continuous $p$th-order derivatives, \textit{i.e.},
\begin{align*}
    \Vert \nabla^p f(\vx) - \nabla^p f(\vy) \Vert \le L_p \Vert \vx- \vy \Vert, \quad \forall \vx, \vy \in \sR^d,
\end{align*}
where $L_p>0$ is the Lipschitz constant of $\nabla^p f$.
\end{asm}

In this paper, we are mostly interested in the case that $p \in \{ 1,2,3\}$, for which we can develop globally convergent second-order algorithms \citep{nesterov2021implementable}.  For the cases $p \ge 4$, although there are conceptual tensor algorithms that could be applied \citep{birgin2017worst,nesterov2021implementable}, it is still open whether the tensor steps in these algorithms can be solved efficiently. Therefore, this work only considers the currently implementable setting that $p \in \{1,2,3\}$. Following \citet{ghadimi2017second,agafonov2024inexact,agafonovadvancing}, we assume that the algorithms have access to a $\delta$-Hessian estimator $\mH(\vx)$, formally defined as follows.

\begin{asm} \label{asm:delta-Hess}
We assume that the algorithm has access to a $\delta$-Hessian estimator $\mH(\vx) \in \sR^{d \times d}$ satisfies that for any $\vx \in \sR^d$, we have  
\begin{align*}
 \mH(\vx) = \mH^\top(\vx) \quad {\rm and} \quad \Vert \mH(\vx) - \nabla^2 f(\vx) \Vert \le \delta.
\end{align*}
\end{asm}

This assumption unifies both first- and second-order methods, which recovers the second-order optimization with exact oracles \citep{nesterov2006cubic}, and first-order optimization when the objective has $L_1$-Lipschitz continuous gradients \citep{nesterov1983method} if $\delta=0$. Under the above assumptions, we aim to find an approximate solution that has a small suboptimality gap:
\begin{dfn}
We say $\hat \vx \in \sR^d$ an $\epsilon$-solution to Problem (\ref{prob:main}) if $f(\hat \vx) - \min_{\vx \in \sR^d} f(\vx) \le \epsilon$.
\end{dfn}
Finally, we count the inexact second-order oracle (ISO) complexity by the number of inexact Hessian $\mH(\vx)$ calls to find an $\epsilon$-solution. 





\section{Accelerated Inexact Newton Extragradient}

\citet{ghadimi2017second,agafonov2024inexact,agafonov2024exploring} proposed inexact second- or higher-order methods by generalizing the Nesterov acceleration \citep{nesterov2008accelerating}. However, in the exact case, the Nesterov acceleration has known suboptimal for high-order optimization since the $\epsilon$ dependence can be improved via the A-NPE framework \citep{monteiro2013accelerated}. Therefore, it is natural to study applying the same technique in the exact setting.




\subsection{Optimal Exact Second-Order Methods} \label{subsec:back-ANPE}

We first review the accelerated Newton proximal extragradient (A-NPE) framework \citep{monteiro2013accelerated} in Algorithm~\ref{alg:A-NPE}, a powerful tool to achieve the optimal convergence rate for second-order methods. 
The most subtle part in Algorithm \ref{alg:A-NPE} is the proximal point step in Eq. (\ref{eq:prox-y-subproblem}), which requires solving an implicit and cyclical nonlinear equation in $(\vy_t,\lambda_{t})$. Below, we briefly recall the solutions to solve this implicit equation in existing literature.

\begin{algorithm}[htbp]  
\caption{A-NPE Framework($\vx_0, T$)}  \label{alg:A-NPE}
\begin{algorithmic}[1] 
\State $A_0 = 0$ 
\State \textbf{for} $t = 0,\cdots, T-1$ 
\State \quad \label{line:proximal-step} Compute $\vy_{t} \in \sR^d$ and $\lambda_{t} \in \sR_+$ such that 
\begin{align} \label{eq:prox-y-subproblem}
 \vy_{t} \approx \arg \min_{\vy \in \sR^d} f(\vy) + \Vert \vy - \vz_t \Vert^2/2,   
\end{align}
\quad ~~where
\begin{align*}
 \vz_t = (A_t \vx_t + a_{t+1}  \vv_t)/{A_{t+1}}, \quad   A_{t+1} = A_t + a_{t+1}, \\
A_{t+1} \lambda_{t}^2 =  a_{t+1}^2 ,~~ {\rm and} ~~\lambda_{t} = \Theta(L_p \Vert \vy_{t} - \vz_t  \Vert^{p-1}).
\end{align*} 
\State \quad Update $\vx_{t+1} = \vy_{t}$ and $\vv_{t+1} = \vv_t - a_{t+1} \nabla f(\vy_{t})$. \label{line:extragrad-step} 
\State \textbf{end for} 
\State \textbf{return} $\vx_T$
\end{algorithmic}
\end{algorithm}

\citet{monteiro2013accelerated} showed that this equation can be solved in $\gO(\log ( L_2/\epsilon))$ iterations by a bisection in $\lambda_{t}$ with Newton updates $ \vy_t = \vz_t -  ( \nabla^2 f(\vz_t) + \lambda_{t} \mI_d)^{-1} \nabla f(\vz_t)$, leading to a near-optimal second-order method with an SO complexity of $\gO((L_2/\epsilon)^{2/7} \log (L_2/\epsilon) )$. A similar bisection strategy is also adopted in the tensor generaliza ons
\citep{gasnikov2019optimal,bubeck2019near,jiang2021optimal} that can achieve a faster rate of 
$\gO((L_3/\epsilon)^{1/5} \log (L_3/\epsilon) )$ under the additional $L_3$-smoothness assumption. However, the bisection could be costly in practice and introduces the unnecessary  additional logarithmic factors in theory.


\citet{carmon2022optimal} removed the bisection using the following strategy. Their method
maintains a guess $\lambda_{t}'$ to calculate $\vz_t$ and then obtains~$\vy_t$ by taking an exact tensor step from $\vz_t$. After $\vy_t$ is computed, it computes $\lambda_t = \Theta(L_p \Vert \vy_{t} - \vz_t  \Vert^{p-1})$ and compares $\lambda_t$ with the guess $\lambda_t'$. If the guess underestimates ($\lambda'_t < \lambda_t$), the algorithm simply performs the same updates as the original A-NPE method \citep{monteiro2013accelerated}. If it overestimates ($\lambda'_t \ge \lambda_t$), which stabilizes the algorithm by scaling down the growth of $A_{t+1}$ by a factor of $\lambda_{t}/ \lambda'_{t}$ and replacing the next iterate $\vx_{t+1}$ by a convex combination of $\vx_t$ and $\vy_t$.
\citet{carmon2022optimal} applies a multiplicative update for the guess $\lambda'_t$. In a concurrent work, \citet{kovalev2022first} obtained the same optimal rate by predetermining the regularization $\lambda_t$ and then solving the proximal subproblem in Eq. (\ref{eq:prox-y-subproblem}) with another Newton extragradient subroutine.

\subsection{Optimal Inexact Second-Order Method} \label{subsec:algo-AINE}
As we see above, one key step in implementing A-NPE is correctly choosing the regularization coefficient $\lambda_t$.
The techniques used in \citet{monteiro2013accelerated,carmon2022optimal,kovalev2022first} are not sufficient to handle the case when the Hessian is inexact. 
The analyses in \citet{monteiro2013accelerated,carmon2022optimal} both require lower-bounding $\Vert \vy_t - \vz_t \Vert^2 $ using $\lambda_t$, which is unfortunately violated by inexact second-order methods where $\lambda_t = \Omega(\delta)$. Moreover, while \citet{kovalev2022first} does not use the aforementioned inequality, it remains unclear how to set the regularization $\lambda_t$ in advance, since it should decrease at a different rate depending on whether $\delta$ is larger than $L_p \Vert \vy_t - \vz_t \Vert^{p-1} / p!$ at each iteration.

We address these issues by leveraging an epoch-wise guessing strategy to stabilize the regularization schedule. Inspired by  \citet{adil2024convex}, our approach partitions iterations into distinct epochs and adaptively assigns a predetermined $\lambda'_t$ via a Guess subroutine. This finally leads to the 
accelerated inexact Newton extragradient (AINE) method as presented in Algorithm \ref{alg:AINE}. The algorithm utilizes the following $\delta$-MS oracle.
\begin{dfn} \label{dfn:delta-MS}
We say the mapping $\sO_p^\delta:  \sR^d \mapsto \sR^d \times \sR_+ $ a $p$th-order $\delta$-MS oracle if for an input point $\vx \in \sR^d$, it returns the a pair $(\vy,\lambda) = \sO_p^\delta(\vx)$ that satisfies:
\begin{enumerate}
    \item The MS condition $ \Vert \nabla f(\vy) + \lambda (\vy - \vx) \Vert \le \lambda \Vert \vy - \vx \Vert /2$.
    \item The $p$th-order $\delta$-inexact movement bound $\lambda \ge c ( \delta + L_p \Vert \vy - \vx \Vert^{p-1} / p! )$ for some $c>0$.
\end{enumerate}
\end{dfn}
The following lemma shows that the output of a $\delta$-inexact CRN step is a $\delta$-MS oracle when $p=2$. 
\begin{lem} \label{lem:inexact-Newton-step}
Under Assumption \ref{asm:lip-pth-deriv} and \ref{asm:delta-Hess},
the second-order $\delta$-MS oracle $\sO_2^\delta: \sR^d \mapsto \sR^d \times \sR_+$ can be implemented by the following inexact cubic regularized Newton (CRN) step:
\begin{align} \label{eq:inexact-CRN-step}
\begin{split}
        \vy =& \arg \min_{\vy \in \sR^d} \langle \nabla f(\vx), \vy - \vx \rangle + \frac{1}{2} \langle (\mH(\vx) + 2\delta \mI_d) (\vy - \vx) , \vy - \vx \rangle + \frac{L_2}{3} \Vert \vy - \vx \Vert^{3}, \\
    \lambda =& 2\delta + L_2 \Vert \vy - \vx \Vert.
\end{split}
\end{align}
\end{lem}
\begin{rmk}
Lemma~\ref{lem:inexact-Newton-step} shows that $\sO_2^{\delta}$ can be achieved by an inexact CRN update for functions with $L_2$-Lipschitz continuous Hessians. 
In the following, we study the convergence of the AINE by assuming access to $\sO_p^{\delta}$, which applies to a broader range of settings by varying the order of $p$. 
We will also show that the proposed framework can lead to a superfast inexact second-order method for functions with $L_3$-Lipschitz continuous third-order derivatives in the next section.
\end{rmk}

\begin{algorithm}[t]  
\caption{AINE($\vx_0, T, \delta, L_p$)}  \label{alg:AINE}
\begin{algorithmic}[1] 
\Require $\delta$-inexact MS oracle $(\vy,\lambda) = \sO_p^\delta(\vz)$ for $p \in \{ 1,2,3\}$.
\Function{Guess}{$\lambda_0, t, A_t$}
\State \textbf{if} $ t= 0 $ \textbf{then}  
\State \quad \textbf{return} $\lambda_t' = \lambda_0$  
\State \textbf{else if} $t =1$ \textbf{or} $ A_t > 2A_s $ \textbf{then}  
\State \quad Store $s = t$ and $A_s = A_t$. \label{Line:store-guess}  
\State \quad \textbf{return} $\lambda'_t$ as Eq. (\ref{eq:para-lambda-s})  
\State \textbf{else}  
\State \quad \textbf{return} $\lambda'_t = \lambda_s'$  
\State \textbf{end if}  
\EndFunction
\State $\vv_0 = \vz_0 = \vx_0$, ~~ $A_0 = 0$  \Comment{Main procedure}
\State $(\vy_0,\lambda_0) = \sO^{\delta}_{p}(\vz_0)$  
\State \textbf{for} $t = 0,\cdots, T-1$ 
\State \quad $\lambda'_{t}  =\text{Guess}(\lambda_0,t,A_t)$   
\State \quad $a_{t+1}' =  ( 1 + \sqrt{1 + 4 \lambda_{t}' A_t})  / (2 \lambda'_{t})$  
\State \quad $A_{t+1}' = A_t + a_{t+1}'$, ~~ $\vz_t = (A_t \vx_t + a_{t+1}' \vv_t)/{A_{t+1}'} $  
\State \quad \textbf{if} $t > 0 $ \textbf{then}   
\State \quad \quad $(\vy_{t}, \lambda_t) = \sO_{p}^{\delta}(\vz_t)$  
\State \quad \textbf{end if}  
\State \quad $\beta_t = \min \{1, \lambda_t'/\lambda_t \}$, ~~ $ a_{t+1} = \beta_t a'_{t+1}$, ~~ $A_{t+1} = A_t + a_{t+1}$  
\State \quad $\hat \vx_{t+1} =  \left( (1- \beta_t ) A_t \vx_t + \beta_t A'_{t+1} \vy_t \right) \big / A_{t+1}$ 
\State \quad \textbf{if} $p \in \{1,2\}$ \textbf{then}
\State \quad \quad $\vx_{t+1} = \hat \vx_{t+1}$ 
\State \quad \textbf{else} \Comment{Use the monotone variant of A-NPE \citep[Line 8, Algo. 4]{carmon2022optimal}.}
\State \quad \quad $\vx_{t+1} \leftarrow \arg \min_{\vx} \{f(\vx) : \vx \in \{ \vx_0, \cdots, \vx_t, \hat \vx_{t+1}  \} \}$ \label{line:monotone-MS-option}
\State \quad \textbf{end if}
\State \quad $\vv_{t+1} = \vv_t - a_{t+1} \nabla f(\vy_{t})$
\State \textbf{end for} 
\State \textbf{return} $\vx_T$ 
\end{algorithmic}
\end{algorithm}

Since the convergence result from \citet{carmon2022optimal} or \citep{adil2024convex} only requires the MS condition (Condition 1 in Definition~\ref{dfn:delta-MS}), and is independent of the other components in Algorithm~\ref{alg:AINE},  we can obtain the following guarantee of the AINE method. 

\begin{lem}[{\citet[Proposition 1]{carmon2022optimal}}]
\label{lem:Car}
Let Assumption \ref{asm:convex} hold and define 
\begin{align} \label{eq:dfn-notation-Car}
\begin{split}
    E_t: = f(\vx_t) - f(\vx^*)~~ {\rm and} ~~ D_t:= \frac{1}{2} \Vert \vv_t - \vx^* \Vert^2.
\end{split}
\end{align}
If in Algorithm \ref{alg:AINE}, all the pair $(\vy_t,\vz_t,\lambda_t)$ for all $t = 0,\cdots, T-1$ satisfies the MS condition (condition 1 of Definition~\ref{dfn:delta-MS}), then the iterates generated by the algorithm satisfy that
\begin{align} \label{eq:decrease-Car}
\begin{split}
    (A_t E_t + D_t) - (A_{t+1} E_{t+1} + D_{t+1}) \ge  \frac{3}{8} A'_{t+1} \min\{ \lambda_{t}, \lambda'_t \} \Vert \vy_t - \vz_t \Vert^2. 
\end{split}
\end{align}
\end{lem}

%

Based on the above lemma, we obtain the following convergence rate of AINE by appropriately selecting the guess in the subroutine. 

\begin{thm} \label{thm:AINE}
Let Assumption \ref{asm:convex}, \ref{asm:lip-pth-deriv}, and \ref{asm:delta-Hess} hold for any $p \in \sN_+$. If we
select $\lambda_s'$ in the Guess subroutine of Algorithm~\ref{alg:AINE} according to 
\begin{align} \label{eq:para-lambda-s}
    \lambda'_s = \Theta \left( \delta + \left( D_0 {L_p^{\frac{2}{p-1}}}/{A_{s}^{\frac{3}{2}}}  \right)^{\frac{2(p-1)}{3p+1}} \right),
\end{align}
then Algorithm \ref{alg:AINE} has the following convergence rate:
\begin{align} \label{eq:final-rate-2nd}
    A_T= \Omega\left( \min \left\{ \frac{T^2}{\delta}, \frac{T^{(3p+1)/2}}{D_0^{(p-1)/2} L_p} \right\} \right) \quad {\rm and} \quad
    E_T \le \frac{D_0}{A_T} = \gO \left( \frac{\delta D_0}{T^2} + \frac{L_p D_0^{(p+1)/2}}{T^{(3p+1)/2}}  \right).
\end{align}
\end{thm}

Below, we provide a proof sketch and demonstrate the difference between prior analyses for optimal exact methods \citep{carmon2022optimal,adil2024convex}. The formal proof is deferred to the appendix.

\begin{proof}[Proof Sketch]
Inequality (\ref{eq:decrease-Car}) in Lemma \ref{lem:Car} indicates that $E_T \le D_0 / A_T$. Therefore, proving the convergence rate of the algorithm suffices to show the growth rate of $A_T$ in inequality (\ref{eq:final-rate-2nd}). Following the analysis in \citet{adil2024convex}, we can give a lower bound of $A_t$ by deriving an upper bound on the number of iterations required for $A_t$ to increase
by a factor of $2$. Let $t_1$ be given and let $t_2$ be the smallest value where $A_{t_2} > 2A_{t_1}$. 
The Guess subroutine fixes $\lambda_{t}' = \lambda_{t_1}' $ in the interval $t \in [t_1,t_2)$.
Let $S^\downarrow$ and $S^{\uparrow}$ be the sets of iterations where $\lambda_t < \lambda'_t$ and its converse  held respectively:
\begin{align*}
    S^{\downarrow} := \{ t \in [t_1,t_2): \lambda_t < \lambda_t' \}~~{\rm and}~~S^{\uparrow} := [t_1,t_2) - S_j^{\rm \downarrow},
\end{align*}
We can show the same upper bound $\vert S^{\downarrow} \vert = \gO( \sqrt{A_{t_1} \lambda_{t_1}'})$ proved by \citet{adil2024convex} still holds for our inexact algorithm. However, their analysis of upper-bounding $\vert S^\uparrow \vert$ fails under Hessian inexactness. Specifically, they
summed-up inequality (\ref{eq:decrease-Car}) to obtain
$$
D_0\geq \frac{3}{8}A_{t+1}'\lambda_t'\|\vy_t-\vz_t\|^2,
$$
then lower-bound $\|y_t-z_t\|$ by $\lambda_t$. However, since $\lambda_t = \delta + L_p/p!\|y_t-z_t\|^{p-1}$ in our setting, it cannot lower-bound $\|y_t-z_t\|$ when $\delta$ is large, which requires a new analysis. To overcome this issue, we develop a novel separation trick to divide $S^{\uparrow}$ by $S_{1}^\uparrow$ and $S_{2}^\uparrow$ to differentiate the cases whether the $\delta$-inexactness or $p$th-order progress dominates in each iteration:
\begin{align*}
    S^\uparrow_\delta:= \{ t \in S^\uparrow: \delta > L_p \Vert \vy_t - \vz_t \Vert^{p-1} / p! \} 
    \quad {\rm and} \quad S^\uparrow_p: = S^\uparrow - S^\uparrow_\delta.
\end{align*}
We then provide different analyses for these two sets individually to obtain the bounds of their sizes:
\begin{align*}
    \vert S^\uparrow_\delta \vert = \gO \left(\delta \sqrt{ \frac{A_{t_1}}{\lambda_{t_1}'}} \right) \quad {\rm and} \quad \vert S^\uparrow_p \vert = \gO \left( \frac{D_0 L_p^{2/(p-1)}}{A_{t_1} (\lambda'_{t_1})^{2p/(p-1)}} \right).
\end{align*}
Finally, we sum up the bounds of $\vert S^\downarrow \vert$, $\vert S_\delta^{\uparrow} \vert$, and $\vert S_p^{\uparrow} \vert$ to obtain the bound of $t_2 - t_1 +1$, and select the guess $\lambda_{t_1}'$ to minimize this upper bound in Eq. (\ref{eq:para-lambda-s}).
\end{proof}
Combining this theorem with Lemma~\ref{lem:inexact-Newton-step}, we can obtain the ISO complexity of AINE for $p=2$.
\begin{cor}[AINE for $p=2$] \label{cor:AINE-p2}
Let Assumption \ref{asm:convex}, \ref{asm:lip-pth-deriv}, and \ref{asm:delta-Hess} hold for any $p =2$. By implementing the second-order $\delta$-MS oracle via the inexact CRN step as Eq. (\ref{eq:inexact-CRN-step}), Algorithm \ref{alg:AINE} can find an $\epsilon$-solution in the inexact second-order oracle (ISO) complexity of $\gO\left( (\delta/ \epsilon)^{1/2} + (L_2/ \epsilon)^{2/7} \right)$.
\end{cor}
\section{Second-Order Implementation of the Third-Order Scheme} 

In this section, we show that the AINE method can be applied beyond Hessian-Lipschitz objectives (Assumption \ref{asm:lip-pth-deriv} with $p=2$), resulting in a faster inexact second-order method under higher-order smoothness (Assumption \ref{asm:lip-pth-deriv} with $p=3$). Our main idea is to generalize the superfast exact second-order method \citep{nesterov2021implementable,nesterov2021inexact,nesterov2021superfast} to inexact oracles. We follow \citep{nesterov2021implementable,nesterov2021inexact,nesterov2021superfast} to establish the relative non-degeneracy condition \citep{bauschke2017descent,lu2018relatively} of a third-order proximal function, then apply a Bregman distance gradient method (BDGM) subroutine \citep{lu2018relatively,nesterov2021implementable,nesterov2021superfast} to solve the proximal subproblem using inexact second-order oracles.

\begin{algorithm}[t]  
\caption{Implement $\sO_3^\delta(\vz_t)$ via BDGM \citep{lu2018relatively,nesterov2021superfast}}  \label{alg:BGDM-for-prox-3}
\begin{algorithmic}[1] 
\State $k=0$, $\vy_{t,0} = \vz_t$
\State Define $f_t$ as Eq. (\ref{eq:prox-func-ft}), $\rho_t$ as Eq. (\ref{eq:scaling-func-rhot}), and $R_t$ as (\ref{eq:dfn-Rt}).
\State \textbf{while} $(\vy_{t,k},\lambda_{t,k})$ does not satisfy the MS condition (part 1 of Definition \ref{dfn:delta-MS})
\State \quad $\vy_{t,k+1} = \arg \min_{\Vert \vy - \vz_t \Vert \le R_t}  \left\{  \langle \nabla f_t(\vy_{t,k}), \vy - \vy_{t,k} \rangle + 2 (1+1/\xi) \beta_\rho(\vy_{t,k},\vy) \right\} $ 
\State \quad $k \leftarrow k+1$
\State \textbf{end while}
\State \textbf{return} $(\vy_t,\lambda_t) = (\vy_{t,k}, \lambda_{t,k})$
\end{algorithmic}
\end{algorithm}

\subsection{Bregman Distance Gradient Method}

This section first introduces the necessary backgrounds on the 
relative non-degeneracy condition \citep{bauschke2017descent,lu2018relatively} and the Bregman distance gradient method (BDGM) subroutine \citep{lu2018relatively,nesterov2021implementable,nesterov2021superfast} for the following constrained convex minimization problem:
\begin{align*}
    \min_{\vy \in \gY}\varphi(\vy).
\end{align*}
Let $\beta_\varphi(\vy,\vy') = \varphi(\vy') - \varphi(\vy) - \langle \nabla \varphi(\vy) , \vy' - \vy \rangle $ be the Bregman distance induced by $\varphi$. The relative non-degeneracy condition and the condition number are formally defined as follows.

\begin{asm} \label{asm:relative}
Assume that $\varphi: \gY \rightarrow \sR$ is non-degenerate with respect to a scaling function $\rho: \gY \rightarrow \sR$:
\begin{align*}
    \mu_\rho(\varphi) \beta_\rho(\vy,\vy') \le \beta_\varphi(\vy,\vy') \le L_\rho(\varphi) \beta_\rho(\vy,\vy'), \quad \forall \vy, \vy' \in \gY,
\end{align*}
where $0 \le \mu_\rho(\varphi) \le L_\rho(\varphi)$. When both $\varphi(\,\cdot\,)$ and $\rho(\,\cdot\,)$ are twice differentiable, the assumption is equivalent to $ \mu_\rho(\varphi) \nabla^2 \rho(\vy) \preceq \nabla^2 \varphi(\vy) \preceq L_\rho(\varphi) \nabla^2 \rho(\vy) $ for all $\vy \in \gY$.
\end{asm}

\begin{dfn} \label{dfn:condition-numer-rho}
Under Assumption \ref{asm:relative}, we denote $\kappa_\rho(\varphi) \triangleq \mu_{\rho}(\varphi) / L_\rho(\varphi)$ as the condition number of the function $\varphi(\,\cdot\,)$ with respect to the scaling function $\rho(\,\cdot\,)$.
\end{dfn}

When the functions satisfy Assumption \ref{asm:relative}, the BGDM updates guaranty a linear convergence rate to the minimizer, as stated in the following lemma.

\begin{lem}[{\citet[Eq. 40]{nesterov2021superfast}}] \label{lem:BDGM}
Under Assumption \ref{asm:relative}, iteratively conducts the update
\begin{align} \label{eq:BDGM-update}
    \vy_{k+1} = \arg \min_{\vy \in \gY} \left\{  \langle \nabla \varphi(\vy_k), \vy - \vy_k \rangle + 2 L_\rho(\varphi) \beta_\rho(\vy_k,\vy) \right\},
\end{align}
ensures a linear convergence to the minimizer of $\vy^* = \arg \min_{\vy \in \gY} \varphi(\vy)$ with the convergence rate:
\begin{align*}
    \min_{0 \le k \le K} \varphi(\vy_k) - \varphi(\vy^*) \le \frac{\mu_\rho(\varphi)}{2} \left( 1- \frac{\kappa_{\rho}(\varphi)}{4} \right)^K \beta_\rho(\vy_0, \vy^*), \quad \forall K \in \sN_+.
\end{align*}
\end{lem}

\subsection{Implementing the Third-Order Inexact MS Oracle}

To implement the third-order inexact MS oracle $\sO_3^{\delta}(\vz_t)$, we first note that Condition 2 of Definition~\ref{dfn:delta-MS} can be automatically satisfied by setting
\begin{align} \label{eq:dfn-lambda-t}
    \lambda_t=   \frac{\xi^2-1}{\xi} \delta + \frac{(1+\xi) \xi L_3}{6} \Vert \vy_t - \vz_t \Vert^{2},
\end{align}
where $\xi>0$ is a numerical constant, and solving the proximal function:
\begin{align} \label{eq:prox-func-ft}
   \min_{\vy \in \sR^d} \left\{f_t(\vy) \triangleq f(\vy) +  \frac{(\xi^2-1) \delta}{2\xi} \Vert \vy - \vz_t \Vert^{2}
    +  \frac{(1+\xi) \xi L_3}{24} \Vert \vy - \vz_t \Vert^{4} \right\}.
\end{align}
Consequently, if we can additionally satisfy (Condition 1 of Definition \ref{dfn:delta-MS}), then we can return a valid MS oracle, which can guarantee the fast convergence rate as a third-order method as indicated by Theorem \ref{thm:AINE}. To fulfill Condition 1 efficiently, we leverage the following relative smoothness.

\begin{lem} \label{lem:relative-smooth}
Under Assumption \ref{asm:convex}, \ref{asm:lip-pth-deriv}, and \ref{asm:delta-Hess} for $p=3$,
$f_t(\vy)$ is non-degenerate with respect to the scaling function:
\begin{align} \label{eq:scaling-func-rhot}
    \rho_t(\vy) = \frac{1}{2} \langle(\mH(\vz_t) + \delta \mI_d) (\vy - \vz_t), \vy - \vz_t \rangle + \frac{(\xi + \xi^2) L_3}{12} \Vert \vy - \vz_t \Vert^4,
\end{align}
and we have $(1 - 1/\xi) \nabla^2 \rho_t(\vy) \preceq \nabla^2 f(\vy) \preceq (1+1/\xi) \nabla^2 \rho_t(\vy)$.
\end{lem}

Although implementing the MS oracle is stating as minimizing the unconstrained optimization\linebreak
$\min_{\vy \in \sR^d} f_t(\vy)$, it is more convenient to work with an equivalent constraint form on a compact set $\gY_t$ for the stability of the process. It is possible by the following lemma.

\begin{lem} \label{lem:prox-funct-ystar-bounded}
Under Assumption \ref{asm:convex}, \ref{asm:lip-pth-deriv}, and \ref{asm:delta-Hess} for $p=3$, for $\vy_t^* = \arg \min_{\vy \in \sR^d} f_t(\vy)$, we have
\begin{align} \label{eq:dfn-Rt}
        \Vert \vz_t - \vy_t^* \Vert \le 2 \left(\frac{12 \Vert \nabla f(\vz_t) \Vert}{L_3 (\xi^2-1)} \right)^{1/3} \triangleq R_t.
\end{align}
\end{lem}

In view of the above lemma, we can replace Problem (\ref{eq:prox-func-ft}) by the following constrained one:
\begin{align} \label{prob:constrained-prox-ft}
    \min_{\vy \in \gY_t} (\vy), \quad {\rm where} \quad \gY_t \triangleq \{ \vy \in \sR^d: \Vert \vy - \vz_t \Vert \le R_t \}
\end{align}
We present the BDGM subroutine applied on Problem (\ref{prob:constrained-prox-ft}) in Algorithm \ref{alg:BGDM-for-prox-3}. This subroutine terminates when $(\vy_k,\lambda_k)$ satisfies the MS condition and then returns $(\vy_t,\lambda_t) = (\vy_k, \lambda_k)$ as the third-order $\delta$-MS oracle for the query point $\vz_t$. And we will present a global complexity bound for each call of the BDGM subroutine in Algorithm~\ref{alg:AINE} in the next section.

\subsection{Global Complexity Bound}
Let $R(\vx_0) = \sup_{\vy \in \sR^d}\{ \Vert \vy - \vx^* \Vert: f(\vy) \le f(\vx_0) \}$ be the size of the sub-level set. The monotone MS acceleration variant we use in Algorithm \ref{alg:AINE} for $p=3$ ensures that all the generated iterates lie in the sub-level set, and then we can use the following lemma to obtain global upper bounds for their gradient and Hessian norms.

\begin{lem}[{\citet[Lemma 5.1]{nesterov2021superfast}}] \label{lem:bound-grad-hess-l3}
If Assumption \ref{asm:lip-pth-deriv} holds for $p=3$, then
for any $\vx \in \{\vx \in \sR^d: \Vert \vx - \vx_0 \Vert \le D \}$, we have
\begin{align} \label{eq:bound-grad-hess-l3}
\begin{split}
\Vert \nabla f(\vx) \Vert \le& \Vert \nabla f(\vx_0) \Vert+ \Vert \nabla^2 f(\vx_0) \Vert + \frac{1}{2} \Vert \nabla^3 f(\vx_0) \Vert D^2 + \frac{1}{6} L_3 D^3, \\
\Vert \nabla^2 f(\vx) \Vert \le& \Vert \nabla^2 f(\vx_0) \Vert + \Vert \nabla^3 f(\vx_0) \Vert D + \frac{1}{2} L_3 D^2.
\end{split}
\end{align}
\end{lem}

Since the size of the constraint set $R_t$ defined in the BDGM subroutine depends on $\Vert \nabla f(\vz_t) \Vert$, which is uniformly bounded by the above lemma, we can conclude the boundedness of all the iterates of our final algorithm for $p=3$.

\begin{lem} \label{lem:bounded-norm}
Under Assumption \ref{asm:convex} and \ref{asm:lip-pth-deriv} for $p=3$, there exists a constant
\begin{align*}
 \bar R = {\rm poly}(L_3, R(\vx_0), \Vert \nabla f(\vx_0) \Vert, \Vert \nabla^2 f(\vx_0) \Vert, \Vert \nabla^3 f(\vx_0) \Vert )   
\end{align*}
such that all the iterates generated by the main procedure in Algorithm \ref{alg:AINE} and the BDGM subroutine (Algorithm~\ref{alg:BGDM-for-prox-3}) lie in the compact set $\gX = \{ \vx \in \sR^d: \Vert \vx - \vx^* \Vert \le \bar R \}$ when implementing the third-order MS oracle $\sO_3^\delta$ 
\end{lem}

Let $\gX \subseteq \sR^d$ be the set defined in Lemma \ref{lem:bounded-norm}, we further define
\begin{align} \label{eq:l0-l1-on-gx}
   \bar L_0 = \sup_{t \in [T], \vy \in \gX} \Vert \nabla f_t(\vy) \Vert \quad {\rm and } \quad \bar L_1 = \sup_{t \in [T], \vy \in \gX} \Vert \nabla^2 f_t(\vy) \Vert.
\end{align}


Then, we show that the relative error of the MS condition (Condition 1 of Definition~\ref{dfn:delta-MS}) can be replaced by $\gO(\delta \Vert \vz_t - \vy_t^* \Vert)$. This relaxed condition is easier to analyze in a subsequent analysis.

\begin{lem} \label{lem:translated-MS-cond}
Let Assumption  \ref{asm:lip-pth-deriv}, \ref{asm:delta-Hess}, and \ref{asm:convex} hold for $p=3$.
For $f_t(\,\cdot\,)$ defined in Eq. (\ref{eq:prox-func-ft}) and $\lambda_t$in Eq.~(\ref{eq:dfn-lambda-t}), the MS condition (part 1 of Definition \ref{dfn:delta-MS}) can be relaxed to
\begin{align} \label{eq:relax-MS-cond}
    \Vert \nabla f_t(\vy_t) \Vert \le \frac{(\xi^2- 1)\delta}{3\xi} \Vert \vz_t - \vy_t^* \Vert.
\end{align}
\end{lem}
Note that we can assume $\Vert \nabla f(\vz_t) \Vert \ge  \epsilon/ R(\vx_0)$ without loss of generality, since otherwise we can conclude from the convexity that $\vz_t$ is already an $\epsilon$-solution. By truncating the iterates to all these non-stationary ones, we can show that the right-hand side of the relaxed condition (\ref{eq:relax-MS-cond}) is lower-bounded and finally arrive at the following key theorem.

\begin{thm} \label{thm:superfast-implementation}
Let Assumption  \ref{asm:lip-pth-deriv}, \ref{asm:delta-Hess}, and \ref{asm:convex} hold for $p=3$.
If $\Vert \nabla f(\vz_t) \Vert \ge \epsilon/ R(\vx_0)$, then the BDGM subroutine (Algorithm \ref{alg:BGDM-for-prox-3}) terminates in the number of iterations no more than
\begin{align} \label{eq:K-log-factor}
    K = {\rm poly} \log(\delta^{-1}, \epsilon^{-1}, L_3, R(\vx_0), \Vert \nabla f(\vx_0) \Vert, \Vert \nabla^2 f(\vx_0) \Vert, \Vert \nabla^3 f(\vx_0) \Vert).
\end{align}
\end{thm}

\begin{cor}[AINE for $p=3$] \label{cor:AINE-p3} 
Let Assumption  \ref{asm:lip-pth-deriv}, \ref{asm:delta-Hess}, and \ref{asm:convex} hold for $p=3$. By implementing the third-order $\delta$-MS oracle via the BDGM subroutine as Algorithm \ref{alg:BGDM-for-prox-3}, Algorithm \ref{alg:AINE} can find an $\epsilon$-solution in the ISO complexity of $\tilde \gO\left( (\delta/ \epsilon)^{1/2} + (L_3/ \epsilon)^{1/5} \right)$, and each oracle call requires running the BDGM subroutine with $K$ iterations, where $K$ is defined in Eq.~(\ref{eq:K-log-factor}).
\end{cor}


\section{Lower Complexity Bound} \label{sec:lb-convex}

We now show that the upper bounds established by Theorem~\ref{thm:AINE} are indeed optimal up to constants by proving matching lower bounds for the following algorithm class.

\begin{dfn} \label{dfn:alg-class}
A deterministic inexact second-order algorithm \texttt{A} consists of a sequence of measurable mappings $\{ \texttt{A}_t \}_{t \in \sN}$ such that $\texttt{A}_t$ takes in the first $t-1$ oracle responses  to produce the $t$th query,  \textit{i.e.}, \texttt{A} generates a sequence $\{\vx_t\}_{t \in \sN}$ such that
\begin{align*}
    \vx_t = \texttt{A}_t ( f(\vx_0), \nabla f(\vx_0), \mH(\vx_0), \cdots, f(\vx_{t-1}), \nabla f(\vx_{t-1}), \mH(\vx_{t-1}) ),
\end{align*}
where $\mH(\vx)$ is the inexact second-order oracle at the query point $\vx$
and satisfies Assumption \ref{asm:delta-Hess}.
\end{dfn}

For the above algorithm class,
we establish lower bounds by using the results in \citet{arjevani2019oracle}, where the $\Omega(L_p/T^{(3p+1)/2})$ lower bound comes from \citet[Theorem 3]{arjevani2019oracle}, and the $\Omega(\delta/T^2)$ lower bound comes from the same theorem with $p = 1$ and $L_1=\delta$.
\begin{prop} \label{prop:lb-convex}
For any $p \in \sN_+$, there exists a function $f:\sR^d \rightarrow \sR$ satisfying Assumption \ref{asm:lip-pth-deriv} and \ref{asm:convex}, such that any algorithm in the class of Definition \ref{dfn:alg-class} has
\begin{align*}
f(\vx_T) - f(\vx^*) = \Omega \left( {\delta D_0}/{T^2} + {L_p D_0^{(p+1)/2}}/{T^{(3p+1)/2}} \right).
\end{align*}
\end{prop}

Note that \citet[Theorem 4.2]{agafonovadvancing} also provides an $\Omega(\delta/T^2+L_2/T^{7/2})$ lower bound for inexact second-order linear span algorithms ($p=2$) using similar reductions, but Proposition~\ref{prop:lb-convex}
generalizes their result to any $p \in \sN_+$ and to all deterministic algorithms. By taking $p=2$ or $p=3$, we can conclude that AINE is optimal in oracle complexity for both settings.

\section{Solving the Auxiliary Subproblems in Nearly Matmul Time}

In this section, we demonstrate that all the auxiliary subproblems in AINE for $p=2$ or $p=3$ can be solved in nearly matrix multiplication time \citep{alman2025more,duan2023faster}. Recall that for $p=2$ we need to solve the inexact CRN step in Eq. (\ref{eq:inexact-CRN-step}) and for $p=3$ we need to conduct the BDGM update in Eq. (\ref{eq:BDGM-update}). We can observe that they can be unified in the following form: 
\begin{align} \label{eq:aux-p2}
\min_{\Vert \vh \Vert \le R} q(\vh) \triangleq \langle \vb , \vh \rangle + \langle \mA \vh, \vh \rangle/2 + c \Vert \vh \Vert^{p+1} / (p+1). 
\end{align}
This subproblem is closely related to the CRN subproblem \citep{nesterov2006cubic} and can be solved in a similar way. By leveraging the eigenvalue decomposition for matrix $\mA$ and then transforming Problem (\ref{eq:aux-p2}) to a one-dimensional root-finding problem, we can show that the subproblem is almost as cheap as a Newton step.

\begin{prop} \label{prop:aux-solver-p}
Let $C$ be the cost of solving the one-dimensional sub-problem and $\omega \approx 2.37$ be the matrix multiplication constant.
For any $p \in \sN_+$, subproblem (\ref{eq:aux-p2}) can be solved in running time of $\gO( d^\omega + Cd)$.
\end{prop}

The above result not only applies to our method, but a similar idea can be used to characterize the running time of the auxiliary problem in \citet{agafonov2024inexact}, as remarked below.

\begin{rmk} \label{rmk:inexact-AITM-aux}
The auxiliary problem when $p=3$ in \citet{agafonov2024inexact} takes the form of $\tilde q(\vh) \triangleq q(\vh) + \lambda \Vert \vh \Vert^3$, which contains a third-order regularization term. However, one can follow the same idea as Proposition \ref{prop:aux-solver-p} to obtain the same running time.
\end{rmk}




\section{Conclusion and Future Works}
In this paper, we propose a novel $\delta$-inexact second-order method AINE for convex optimization. We show that the proposed method achieve an optimal ISO complexity of $\gO((\delta/\epsilon)^{1/2} + (L_2/\epsilon)^{2/7})$ for functions with $L_2$-Lipschitz continuous Hessians and that of $\gO((\delta/\epsilon)^{1/2} + (L_3/\epsilon)^{1/5})$ with $L_3$-Lipschitz third-order derivatives.  Our method is tight for both setups, and we hope that it can be a starting point to study many practical inexact/approximate second-order methods in the future. We discuss these potential future directions in Appendix \ref{apx:future-work}.

\bibliography{references}
\bibliographystyle{plainnat}

\newpage
\appendix

\input{apx}

\end{document}

%% file: math.tex
\usepackage{amsmath,amsfonts,bm,amsthm,amssymb}

\def\eqref#1{equation~\ref{#1}}

\def\1{\bm{1}}

\def\vzero{{\bm{0}}}

\def\va{{\bm{a}}}
\def\vb{{\bm{b}}}
\def\vc{{\bm{c}}}

\def\vh{{\bm{h}}}

\def\vs{{\bm{s}}}

\def\vv{{\bm{v}}}

\def\vx{{\bm{x}}}
\def\vy{{\bm{y}}}
\def\vz{{\bm{z}}}

\def\mA{{\bm{A}}}

\def\mH{{\bm{H}}}
\def\mI{{\bm{I}}}

\def\mO{{\bm{O}}}

\DeclareMathAlphabet{\mathsfit}{\encodingdefault}{\sfdefault}{m}{sl}
\SetMathAlphabet{\mathsfit}{bold}{\encodingdefault}{\sfdefault}{bx}{n}

\def\gO{{\mathcal{O}}}

\def\gX{{\mathcal{X}}}
\def\gY{{\mathcal{Y}}}

\def\sB{{\mathbb{B}}}

\def\sN{{\mathbb{N}}}
\def\sO{{\mathbb{O}}}

\def\sR{{\mathbb{R}}}

\newtheorem{thm}{Theorem}[section]
\newtheorem{dfn}{Definition}[section]

\newtheorem{lem}{Lemma}[section]
\newtheorem{asm}{Assumption}[section]
\newtheorem{rmk}{Remark}[section]
\newtheorem{cor}{Corollary}[section]

\newtheorem{prop}{Proposition}[section]

%% file: apx.tex
\section{Additional Related Works} \label{sec:related}

\paragraph{Efficient Second-Order Methods.} Second-order methods are known for fast convergence rates at a cost of higher computation overhead compared to first-order methods. It is a highly active research area for developing efficient second-order methods. Many studies focuses on different aspects from this work, including but no limited to: quasi-Newton methods \citep{broyden1970convergence,conn1991convergence,rodomanov2022rates,jin2023non,jiang2023accelerated,jiang2023online,jiang2025improved}, stochastic Newton methods \citep{zhou2019stochastic,arjevani2020second,antonakopoulos2022extra}, subspace Newton methods \citep{hanzely2020stochastic,jiang2024krylov}, lazy Hessian updates \citep{doikov2023second,adil2025balancing}, and efficient implementations in neural networks \citep{grosse2016kronecker,gupta2018shampoo,liu2024sophia,vyas2025soap,abreu2025potential}. These works are complementary to ours with orthogonal techniques that can be jointly applied.

\paragraph{Accelerated Inexact Second-Order Methods.} Although extra-gradient acceleration \citep{monteiro2012iteration,monteiro2013accelerated,carmon2022optimal,kovalev2022first} is widely recognized as a standard technique to obtain optimal oracle complexities in exact second-order or higher-order optimization, it is challenging to generalize these  optimal exact methods to the inexact setting. We briefly introduce the results in closely related works to demonstrate the technical challenges to balance the convergence in $\delta$-inexact term and the exact one: \citet{antonakopoulos2022extra} proposed a stochastic and adaptive second-order method Extra-Newton by combining Newton updates in extra-gradient methods, but their $\gO((\delta/\epsilon)^{2/3} + (L_2/\epsilon)^{1/3} )$ ISO complexity has a high dependency on the Hessian inexactness $\delta$ even in the deterministic case.
\citet{jiang2023accelerated} studied a quasi-Newton generalization of A-NPE, but directly applying their analysis to our setting can only lead to a suboptimal ISO complexity of $\gO( (\delta/\epsilon)^{1/2} + (L_2/ \epsilon)^{2/5})$. Compared with them, this work aims at establishing much lower and optimal complexity bounds on both convergent terms by the new analysis in Section \ref{subsec:algo-AINE}.

\section{Future Works} \label{apx:future-work}

There are several future directions to extend our work. First, this work only considers a uniform inexactness $\delta$ in Hessian estimator. It is interesting to consider iteration-dependent inexactness $\delta_t$ and making our framework covers broader Hessian estimators, including quasi-Newton estimators \citep{broyden1970convergence,conn1991convergence}, lazy Hessian estimators \citep{doikov2023second,chayti2023unified,adil2025balancing}, and variance-reduced estimators \citep{zhou2019stochastic,zhou2020bstochastic,zhou2022dimension,wang2019stochastic,emmenegger2022oracle}. Second, it is also important to study stochastic \citep{antonakopoulos2022extra} and nonconvex problems \citep{nesterov2006cubic,arjevani2020second,tripuraneni2018stochastic}. Third, our methods may be extended to solve variational inequalities \citep{agafonov2024exploring} using $\delta$-inexact Jacobians or minimax optimization. 
\section{Proof of Theorem \ref{thm:AINE}}


\begin{proof}
By inequality (\ref{eq:decrease-Car}) of Lemma \ref{lem:Car}, we have $E_t \le D_0 / A_t$ for all $t \in [T]$. Therefore, proving the convergence rate suffices to estimate the growth of $A_t$.

Following the idea in \citet{adil2024convex}, we denote $\{t_j\}$ as the set of the timestamps $t_j$, where the condition $t_j = 1$ or $A_{t_j} > 2 A_{t_{j-1}}$ in the Guess subroutine of Algorithm \ref{alg:AINE} happens. Then, 
\begin{align*}
    1 = A_{t_1} < 2 A_{t_2} < \cdots < 2^j A_{t_{j}}.
\end{align*}
This divides all iterations into multiple epochs, where the $j$th epoch is defined as $e_j :=\{ t \in \sN_+ : t_{j} \le t < t_{j+1}\}$. Next, we give an upper bound on each epoch length $\vert e_j \vert$. We further divide each epoch by two distinct sets, depending on whether $ \lambda_t \le \lambda_t' $ happens in Algorithm \ref{alg:AINE}. Formally, we let
\begin{align*}
    S_j^{\downarrow} := \{ t \in e_j: \lambda_t < \lambda_t' \} \quad {\rm and} \quad S_j^{\uparrow} := e_j - S_j^{\rm \downarrow}.
\end{align*}
Then we separately consider the iterates in $S_{j}^{\rm \downarrow}$ and in $S_j^{\rm \uparrow}$. For $t \in S_{j}^{\downarrow}$, we have $\lambda_t < \lambda_t'$ and the momentum damping mechanism in Algorithm \ref{alg:AINE} is not triggered, which means $\beta_t = 1$, $a_{t+1} = a'_{t+1}$, $A_{t+1} = A'_{t+1}$. In this case, we have
\begin{align*} 
\begin{split}
  &\sqrt{A_{t+1}} - \sqrt{A_t} = \frac{A_{t+1} - A_t}{\sqrt{A_{t+1}}+ \sqrt{A_t}} = \frac{a_{t+1}}{\sqrt{A_{t+1}}+ \sqrt{A_t}} =  \frac{\sqrt{A_{t+1} / \lambda'_{t}}}{\sqrt{A_{t+1}}+ \sqrt{A_t}} \ge \frac{1}{2 \sqrt{\lambda'_{t}}},
\end{split}
\end{align*}
where the last equality uses the fact that $a_{t+1} = a'_{t+1}$ satisfies the equation $ \lambda_t' (a'_{t+1})^2 = A_t + a'_{t+1} = A'_{t+1}$. Telescoping this inequality for $t \in S_j^{\downarrow}- \{t_{j+1}-1\}$ and noting that $\lambda'_t = \lambda'_{t_j}$ is fixed in the epoch, we have
\begin{align} \label{eq:ub-down}
    (\sqrt{2}-1) \sqrt{A_{t_j}}  \ge \sqrt{A_{t_{j+1}-1}} - \sqrt{A_{t_j}} \ge \frac{\vert S_j^{\downarrow} \vert-1}{2 \sqrt{\lambda'_{t_j}}}.
\end{align}
For $t \in S_j^{\uparrow}$, we have $\lambda_t \ge \lambda'_t$ and the momentum damping mechanism in Algorithm \ref{alg:AINE} is triggered, then $\beta_t = \lambda_t' / \lambda_t$, $a_{t+1} = \beta_t a'_{t+1}$, $A_{t+1} = A_t + \beta_t a'_{t+1}$. In this case, we further divide $ S_j^{\uparrow}$ by two smaller distinct sets, depending on whether $\delta > L_p \Vert \vy_t - \vz_t \Vert^{p-1} /p!$ happens. Formally, we let 
\begin{align*}
    &S_{j,\delta}^{\uparrow} := \{ t \in S_j^{\uparrow} : \delta> L_p \Vert \vy_t - \vz_t \Vert^{p-1} / p! \} \quad {\rm and} \quad S_{j,p}^{\uparrow} := S_j^{\uparrow}- S_{j,1}^{\uparrow}.
\end{align*}
For $t \in S_{j,\delta}^{\uparrow}$, we have
\begin{align*}
\frac{1}{\lambda_t} = 
\frac{1}{2(\delta+ L_p \Vert\vy_t - \vz_t \Vert^{p-1}/p!)} > \frac{1}{4\delta}.
\end{align*}
Combing this inequality with $a_{t+1} = a'_{t+1} \lambda_t' / \lambda_t $ and
$a'_{t+1} \ge \sqrt{A_t/\lambda'_t}$, we have
\begin{align*}
    a_{t+1} = \frac{\lambda_t' a'_{t+1}}{\lambda_t} \ge \frac{\sqrt{\lambda'_t A_t}}{\lambda_t} >\frac{\sqrt{\lambda'_t A_{t}}}{4 \delta}.
\end{align*}
Telescoping this inequality for $t \in S_{j,\delta}^{\uparrow}- \{t_{j+1}-1\}$ and noting that $\lambda'_t = \lambda'_{t_j}$ is fixed and $A_t \ge A_{t_j}$ in the $j$th epoch gives
\begin{align} \label{eq:ub-up-1}
A_{t_j} \ge A_{t_{j+1}-1} - A_{t_j}> \frac{\sqrt{\lambda'_{t_j} A_{t_j}} (\vert S_{j,\delta}^{\uparrow} \vert -1 )}{4\delta}.
\end{align}
For $t \in S_{j,p}^{\uparrow}$, we have $\lambda_t \ge \lambda'_t$ and
\begin{align*}
   \Vert \vy_t - \vz_t \Vert =& \left( \frac{p!}{2L_p} \right)^{\frac{1}{p-1}} \left( \frac{2L_p \Vert \vy_t - \vz_t \Vert^{p-1}}{p!} \right)^{\frac{1}{p-1}} \\
   \ge&  \left( \frac{p!}{2L_p} \right)^{\frac{1}{p-1}} \left( \delta+ \frac{L_p \Vert \vy_t - \vz_t \Vert^{p-1}}{p!} \right)^{\frac{1}{p-1}} \\
   =&  \left( \frac{p!}{4L_p} \right)^{\frac{1}{p-1}} \lambda_t^{\frac{1}{p-1}} \ge  \left( \frac{p!}{4L_p} \right)^{\frac{1}{p-1}} \left(\lambda'_t\right)^{\frac{1}{p-1}}.
\end{align*}
Then telescoping inequality~(\ref{eq:decrease-Car}) for $t \in S_{j,p}^{\uparrow} - \{t_{j+1} -1\}$ with $\sigma=1/2$ gives
\begin{align}  \label{eq:ub-up-2}
\begin{split}
    D_0 \ge& \frac{3}{8} \sum_{t=t_j}^{t_{j+1}-1} A'_{t+1} \lambda'_t \Vert \vy_t - \vz_t \Vert^2 
    \\
    \ge& \frac{3}{8} \left( \frac{p!}{4L_p} \right)^{\frac{2}{p-1}}  \sum_{t=t_j}^{t_{j+1}-1} A'_{t+1} \left(\lambda'_t \right)^{\frac{2p}{p-1}} \\
    \ge& 
    \frac{3}{8} \left( \frac{p!}{4L_p} \right)^{\frac{2}{p-1}} (\vert S_{j,2}^\uparrow \vert -1) A_{t_j}
    \left(\lambda'_{t_j} \right)^{\frac{2p}{p-1}}.
\end{split}
\end{align}
Note that $\vert e_j \vert = \vert S_j^\downarrow \vert + \vert S_{j,\delta}^\uparrow \vert + \vert S_{j,p}^\uparrow \vert$, we can combine inequalities (\ref{eq:ub-down}), (\ref{eq:ub-up-1}), and (\ref{eq:ub-up-2}) to obtain 
\begin{align} \label{eq:ub-ei-total}
    \vert e_j \vert = \gO\left(\sqrt{A_{t_j} \lambda'_{t_j}} + \delta \sqrt{ \frac{A_{t_j}}{\lambda_{t_j}'}} + \frac{D_0 L_p^{2/(p-1)}}{A_{t_j} (\lambda'_{t_j})^{2p/(p-1)}}  \right).
\end{align}
Since the value of $A_{t_j}$ is already known at the beginning of the epoch, we can pick $\lambda'_{t_j}$ to minimize the epoch length~$\vert e_j \vert$ in Eq. (\ref{eq:ub-ei-total}).
Formally, we choose
\begin{align*}
    \lambda'_{t_j} = \Theta \left(  \delta  + \left( {D_0 L_p^{\frac{2}{p-1}}}/{A_{t_j}^{\frac{3}{2}}}  \right)^{\frac{2(p-1)}{3p+1}}\right),
\end{align*}
which corresponds to the setting of Eq. (\ref{eq:para-lambda-s}). Then we have
\begin{align*}
    \vert e_j \vert = \gO \left( \sqrt{\delta A_{t_j}} +  D_0^{\frac{p-1}{3p+1}} ( L_p A_{t_j} )^{\frac{2}{3p+1}}   \right).
\end{align*}
Now, summing up all the epochs and using the fact that $A_t \le A_T$ for all $t \in [T]$ gives
\begin{align*}
    T \le \gO\left( \sqrt{\delta A_T} +D_0^{\frac{p-1}{3p+1}} ( L_p A_T )^{\frac{2}{3p+1}} \right).
\end{align*}
Conversely, it means that $A_T$ must satisfy
\begin{align*}
    A_T = \Omega\left( \min \left\{ \frac{T^2}{\delta},
    \frac{T^{(3p+1)/2}}{D_0^{(p-1)/2} L_p} \right\} \right).
\end{align*}
Finally, using $E_T \le D_0 / A_T$ completes the proof.
\end{proof}

\section{Proof of Lemma \ref{lem:inexact-Newton-step}}

\begin{proof}
Recall that $\lambda = 2 \delta + L_2 \Vert \vy - \vx \Vert$.
From the first-order optimality condition of Eq.~(\ref{eq:inexact-CRN-step}), we have 
\begin{align*}
    \vzero = \nabla f(\vx) + (\mH(\vx) + 2 \delta \mI_d) (\vy - \vx) + 2L_3 \Vert \vy - \vx \Vert (\vy - \vx). 
\end{align*}
Rearranging this inequality gives 
\begin{align*}
     &\nabla f(\vy) + \lambda (\vy - \vx) \\
     =& \nabla f(\vy) - \nabla f(\vx) - \mH(\vx)(\vy-\vx) \\
     =& \nabla f(\vy) - \nabla f(\vx) - \nabla^2(\vx)(\vy-\vx) + (\nabla^2 f(\vx) - \mH(\vx))(\vy - \vx).
\end{align*}
Hence, taking the norm on both sides and using the triangle inequality, we obtain 
\begin{align*}
    &\Vert \nabla f(\vy) + \lambda (\vy - \vx) \Vert \\
    \le& \Vert \nabla f(\vy) - \nabla f(\vx) - \nabla^2(\vx)(\vy-\vx) \Vert + \Vert \nabla^2 f(\vx) - \mH(\vx) \Vert \Vert \vy - \vx \Vert \\
    \le& \frac{L_2}{2} \Vert \vy - \vx \Vert^2 + \delta \Vert \vy - \vx \Vert  = \frac{\lambda}{2} \Vert \vy - \vx \Vert,
\end{align*}
where the last inequality uses Assumption \ref{asm:lip-pth-deriv} and \ref{asm:delta-Hess}.
\end{proof}

\section{Proof of Lemma \ref{lem:relative-smooth}}

\begin{proof}

Following \citet{agafonovadvancing,nesterov2021implementable}, we denote the $i$th-order directional derivative of function $f:\sR^d \rightarrow \sR$ at the point $\vx \in \sR^d$ along with directions $\vs_1,\cdots,\vs_i \in \sR^d$ as $\nabla^i f(\vx) [ \vs_1,\cdots, \vs_d]$. For simplicity, we also denote $\nabla^i f(\vx) [\vs]^i$ if $\vs_1=\cdots=\vs_i$.
Recall that when $f:\sR^d \rightarrow \sR$ is a convex function, we have the matrix inequality \citep[Lemma 3]{nesterov2021implementable}:
\begin{align*}
    - \frac{1}{\xi} \nabla f(\vy) - \frac{\xi L_3}{2} \Vert \vy - \vz_t \Vert^2 \preceq \nabla^3 f(\vz_t)[\vy -\vz_t] \preceq \frac{1}{\xi} \nabla f(\vy) - \frac{\xi L_3}{2} \Vert \vy - \vz_t \Vert^2, \quad \forall \xi >0.
\end{align*}
Then, using Assumption \ref{asm:lip-pth-deriv} ($p=3$) and \ref{asm:delta-Hess}, we have
\begin{align*}
    \nabla^2 f(\vy) \preceq& 
    \nabla^2 f(\vz_t) + \nabla^3 f(\vz_t)[\vy -\vz_t] + {L_3} \Vert \vy - \vz_t \Vert^2 \mI_d \\
    \preceq& \left( 1 + \frac{1}{\xi} \right) \nabla^2 f(\vz_t) + (1+\xi) L_3 \Vert \vy - \vz_t \Vert^2 \mI_d \\
    \preceq& \left( 1 + \frac{1}{\xi} \right) (\mH(\vz_t) + \delta \mI_d) + (1+\xi) L_3 \Vert \vy - \vz_t \Vert^2 \mI_d.
\end{align*}
and
\begin{align*}
    \nabla^2 f(\vy) \succeq&  \nabla^2 f(\vz_t) + \nabla^3 f(\vz_t)[\vy -\vz_t] - {L_3} \Vert \vy - \vz_t \Vert^2 \mI_d \\
    \succeq & \left( 1 - \frac{1}{\xi} \right) \nabla^2 f(\vz_t) -(1+\xi) L_3 \Vert \vy - \vz_t \Vert^2 \mI_d \\
    \succeq& \left( 1 - \frac{1}{\xi} \right) (\mH(\vz_t) - \delta \mI_d) - (1+\xi) L_3 \Vert \vy - \vz_t \Vert^2 \mI_d.
\end{align*}
Therefore, for the proximal function $f_t(\vy)$ we define in Eq. (\ref{eq:prox-func-ft}) and the scaling function $\rho_t(\vy)$ in Eq. (\ref{eq:scaling-func-rhot}), we have
\begin{align*}
    \nabla^2 f_t(\vy) \preceq & \left( 1 + \frac{1}{\xi} \right) \left(\mH(\vz_t) + \xi \delta \mI_d + (\xi+\xi^2) L_3 \Vert \vy - \vz_t \Vert^2 \mI_d \right) = \left( 1 + \frac{1}{\xi} \right)  \nabla^2 \rho_t(\vy)
\end{align*}
and 
\begin{align*}
\nabla^2 f_t(\vy) \succeq& \left( 1 - \frac{1}{\xi} \right) (\mH(\vz_t) + \xi \delta \mI_d + (\xi+\xi^2) L_3 \Vert \vy - \vz_t \Vert^2 \mI_d) = \left( 1 - \frac{1}{\xi} \right) \nabla^2 \rho_t(\vy).
\end{align*}

\end{proof}

\section{Proof of Lemma \ref{lem:prox-funct-ystar-bounded}}

\begin{proof}
The proof of this lemma follows the same idea as \citet[Lemma 4.3]{nesterov2021superfast}. Let $d_4(\vy) = \Vert \vy \Vert^4 /4 $ be the prox-function. We know that this function satisfies the following inequality \citep[inequality 2.3]{nesterov2021superfast}:
\begin{align} \label{eq:uniformly-convex-d4}
    d_4(\vy') \ge d_4(\vy) + \langle \nabla d_4(\vy), \vy' - \vy \rangle + \frac{1}{16} \Vert \vy' - \vy \Vert^4, \quad \forall \vy,\vy' \in \sR^d. 
\end{align}
Therefore, for the proximal function $f_t(\vy)$ we define in Eq. (\ref{eq:prox-func-ft}) and the scaling function $\rho_t(\vy)$ in Eq. (\ref{eq:scaling-func-rhot}), we have
\begin{align*}
    & \langle \nabla f(\vz_t), \vz_t - \vy_t^* \rangle = \langle \nabla f_t(\vz_t), \vz_t - \vy_t^* \rangle = f_t(\vz_t) - f_t(\vz_t^*) + \beta_{f_t}(\vz_t, \vy_t^*)\\ 
    =&\beta_{f_t}(\vy_t^*, \vz_t) + \beta_{f_t}(\vz_t, \vy_t^*)
    \ge \left(1 - \frac{1}{\xi}\right) \left( \beta_{\rho_t}(\vy_t^*, \vz_t) + \beta_{\rho_t}(\vz_t, \vy_t^*) \right) \\
    \ge& \left(1 - \frac{1}{\xi}\right) \frac{(\xi + \xi^2) L_3}{12} \left( \beta_{d_4}(\vy_t^*, \vz_t) + \beta_{d_4}(\vz_t, \vy_t^*) \right) \\
    \ge&  \left(1 - \frac{1}{\xi}\right) \frac{(\xi + \xi^2) L_3}{96} \Vert \vz_t - \vy_t^* \Vert^4,
\end{align*}
where the last step uses inequality (\ref{eq:uniformly-convex-d4}). Finally, rearranging the above inequality and applying the Cauchy-Schwarz inequality leads to
\begin{align*}
    \Vert \vz_t - \vy_t^* \Vert \le 2 \left(\frac{12 \Vert \nabla f(\vz_t) \Vert}{L_3 (\xi^2-1)} \right)^{1/3} \le 2 \left(\frac{12L_0}{L_3 (\xi^2-1)} \right)^{1/3}.
\end{align*}
\end{proof}

\section{Proof of Lemma \ref{lem:bounded-norm}}

\begin{proof}
For the monotone variant of AINE we use in Algorithm \ref{alg:AINE} for $p=3$, we have $f(\vx_t) \le f(\vx_0)$ and hence $\Vert \vx_t - \vx^* \Vert \le R(\vx_0)$. Then, Lemma \ref{lem:Car} indicates that $\Vert \vv_t - \vx^* \Vert \le \Vert \vx_0 - \vx^* \Vert \le R(\vx_0)$. Therefore, we have $\Vert \vz_t - \vx^* \Vert \le \max\{ \Vert \vx_t - \vx^* \Vert, \Vert \vv_t - \vx^* \Vert \Vert \} \le R(\vx_0)$. Therefore, all the iterates in the outer loop of Algorithm~\ref{alg:AINE} lies in the compact set $\{\vx \in \sR^d: \Vert \vx - \vx^* \Vert \le R(\vx_0)\}$.

Now, we show that all the iterates generated by the BDGM subroutine in the inner loop to implement the MS oracle also lies in a compact set. By Lemma \ref{lem:prox-funct-ystar-bounded}, it remains to upper-bound $R_t$ defined in the lemma, which requires an upper bound of $\Vert \nabla f(\vz_t) \Vert$. Since $\Vert \vz_t - \vx_0 \Vert \le \Vert \vz_t - \vx^* \Vert + \Vert \vx_0 - \vx^* \Vert \le 2R(\vx_0)$, we can apply Lemma \ref{lem:bound-grad-hess-l3} to upper-bound $\Vert \nabla f(\vz_t) \Vert$ by
\begin{align*}
    \Vert \nabla f(\vz_t) \Vert \le& \Vert \nabla f(\vx_0) \Vert+ \Vert \nabla^2 f(\vx_0) \Vert + 2 \Vert \nabla^3 f(\vx_0) \Vert R^2(\vx_0) + \frac{4L_3 R^3(\vx_0)}{3}.
\end{align*}
Hence, $R_t = {\rm poly}(L_3, R(\vx_0), \Vert \nabla f(\vx_0) \Vert, \Vert \nabla^2 f(\vx_0) \Vert, \Vert \nabla^3 f(\vx_0) \Vert )$ and so does $\bar R$.
\end{proof}

\section{Proof of Lemma \ref{lem:translated-MS-cond}}

\begin{proof}
Starting with the relaxed condition (\ref{eq:relax-MS-cond}), we have
\begin{align*}
     \Vert \nabla f_t(\vy_t) \Vert \le& \frac{(\xi^2- 1)\delta}{3\xi} (\Vert \vz_t - \vy_t \Vert + \Vert \vy_t - \vy_t^* \Vert) \\
     \le&  \frac{(\xi^2- 1)\delta}{3\xi} \Vert \vz_t - \vy_t \Vert + \frac{1}{3}\Vert \nabla f_t(\vy_t) \Vert,
\end{align*}
where the second inequality uses the fact that $f_t(\,\cdot\,)$ defined in Eq. (\ref{eq:prox-func-ft}) is $(\xi^2-1) \delta/\xi$-strongly convex and therefore $(\xi^2-1) \delta/\xi \Vert \vy - \vy_t^* \Vert \le \Vert \nabla f_t(\vy) \Vert$. Rearranging the above inequality, we obtain that
\begin{align*}
    \Vert \nabla f_t(\vy_t) \Vert \le& \frac{(\xi^2- 1)\delta}{2\xi} \Vert \vz_t - \vy_t \Vert \le \frac{\lambda_t}{2} \Vert \vz_t - \vy_t \Vert,
\end{align*}
where the second inequality is due to $ \lambda_t \ge (\xi^2-1) \delta/\xi$ from the the definition of $\lambda_t$ in Eq. (\ref{eq:dfn-lambda-t}).
\end{proof}

\section{Proof of Theorem \ref{thm:superfast-implementation}}

\begin{proof}
Under the notations in this theorem, Lemma \ref{lem:BDGM} can be restated as
\begin{align} \label{eq:BDGM-yt}
    f_t(\vy_t) - f_t(\vy_t^*) \le \underbrace{\frac{\xi-1}{2\xi} \left( 1 - \frac{\xi+1}{4(\xi-1)} \right)^K}_{\alpha_K} \beta_{\rho_t}(\vz_0,\vy_t^*).
\end{align}
For the smoothness constant $\bar L_1$ defined in Eq. (\ref{eq:l0-l1-on-gx}), if $\Vert \nabla f(\vz_t) \Vert \ge \bar \epsilon$, then we have
\begin{align*}
 \Vert \vz_t - \vy_t^* \Vert \ge \frac{1}{\bar L_1}
    \Vert \nabla f_t(\vz_t) \Vert =\frac{1}{\bar L_1} \Vert \nabla f(\vz_t) \Vert \ge \frac{\bar \epsilon}{\bar L_1}.
\end{align*}
Therefore, by Lemma \ref{lem:translated-MS-cond}, to satisfy the MS condition, it suffices to find $\vy_t \in \sR^d$ such that
\begin{align} \label{eq:MS-condition-relax-twice}
    \Vert \nabla f_t(\vy_t) \Vert \le \frac{(\xi^2-1) \delta \bar \epsilon}{3 \xi \bar L_1}. 
\end{align}
Now, let us show that the BDGM subroutine (Algorithm \ref{alg:BGDM-for-prox-3}) can efficiently find such $\vy_t \in \sR^d$.  We have
\begin{align*}
    \frac{1}{16 \bar L_1^4} \Vert \nabla f_t(\vy_t) \Vert^{4} \le \frac{1}{16} \Vert \vy_t - \vy_t^* \Vert^4 \le f_t(\vy_t) - f_t(\vy_t^*) \le \alpha_K \beta_{\rho_t}(\vz_t,\vy_t^*).
\end{align*}
where the first inequality uses the fact that $f_t(\vy)$ has $\bar L_1$-Lipschiitz continuous gradients on the set $\gX$, the second inequality uses inequality (\ref{eq:uniformly-convex-d4}), the third one is due to inequality (\ref{eq:BDGM-yt}). Note that $\Vert \vz_t - \vy_t^* \Vert \le R_t$ by Lemma \ref{lem:prox-funct-ystar-bounded}. We can then apply the norm-dominance inequality \citep[Lemma 3.2]{nesterov2021superfast} to obtain that
\begin{align*}
    \beta_{\rho_t}( \vz_t,\vy_t^* ) \le \frac{\xi^2-1}{2\xi} \left(\bar L_1 + 2\delta  \right) R_t^2 + \frac{(1+\xi)\xi L_3^2 R_t^4}{2}.
\end{align*}
Recall that $R_t = {\rm poly}(L_3, R(\vx_0), \Vert \nabla f(\vx_0) \Vert, \Vert \nabla^2 f(\vx_0) \Vert, \Vert \nabla^3 f(\vx_0) \Vert )$. We then know that the MS condition can be satisfied with $K = {\rm poly} \log(\delta^{-1}, \bar \epsilon^{-1}, \bar L_1, L_3, \sup_{t \in [T]} R_t)$.
\end{proof}

\section{Proof of Proposition \ref{prop:lb-convex}}
\begin{proof}
Note that a lower bound of 
$\Omega\left(L_p D_0^{(p+1)/2}/ T^{(3p+1)/2}\right)$ follows immediately from \citet[Theorem~3]{arjevani2019oracle}. Therefore, it suffices to prove a lower bound of $\Omega(\delta D_0/{T^2})$. This can can be obtained by applying the \citet[Theorem~3]{arjevani2019oracle} with $p=1$ and $L_1= \delta$, whose hard instance $f : \sR^d \rightarrow \sR$ is quadratic and therefore has $0$-Lipschitz continuous $p$th-order derivatives for any $p \ge 2$. As for the Hessian estimator $\mH(\vx)$, we can simply set it as $\mH(\vx) = \mO_d$ to satisfy Assumption~\ref{asm:delta-Hess} because we have $\Vert \mH(\vx) - \nabla^2 f(\vx) \Vert = \Vert \nabla^2 f(\vx) \Vert \le \delta$ under such a setting.
\end{proof}

\section{Proof of Proposition \ref{prop:aux-solver-p}}

\begin{proof}

For subproblem (\ref{eq:aux-p2}) of the form:
\[
\min_{\Vert \vh \Vert \le R} q(\vh) \triangleq \langle \vb , \vh \rangle + \frac{1}{2} \langle \mA \vh, \vh \rangle + \frac{c}{p+1} \Vert \vh \Vert^{p+1},
\]
we consider two cases depending on whether the optimal solution $\vh^*$ lies on the boundary such that $\Vert \vh^* \Vert = R$. If not, then $\boldsymbol{h}^*$ lies in the interior of the feasible set ($\|\boldsymbol{h}\| < R$), it must satisfy the first-order optimality condition:
$$\vzero =  \boldsymbol{b} + \boldsymbol{A}\boldsymbol{h} + c \|\boldsymbol{h}\|^{p-1} \boldsymbol{h}. 
$$

Let $\lambda = c \|\boldsymbol{h}\|^{p-1}$ be a non-negative scalar. The equation becomes $(\boldsymbol{A} + \lambda \boldsymbol{I})\boldsymbol{h} = -\boldsymbol{b}$. For any fixed $\lambda$, the solution is $\boldsymbol{h}(\lambda) = -(\boldsymbol{A} + \lambda \boldsymbol{I})^{-1}\boldsymbol{b}$. The problem reduces to finding $\lambda$ such that:
$$\lambda = c  \|\boldsymbol{h}(\lambda)\|^{p-1}.$$

Otherwise, if  the solution $\vh^*$ lies on the boundary $\|\boldsymbol{h}\| = R$. The KKT conditions imply the existence of a Lagrange multiplier $\sigma > 0$ such that:
$$\vzero = \boldsymbol{b} + \boldsymbol{A}\boldsymbol{h} + c R^{p-1} \boldsymbol{h} + \sigma \boldsymbol{h}.$$
By defining a combined multiplier $\gamma = c R^{p-1} + \sigma$, we obtain the same functional form: $\boldsymbol{h}(\gamma) = -(\boldsymbol{A} + \gamma \boldsymbol{I})^{-1}\boldsymbol{b}$. Now, we must find $\gamma$ such that the norm constraint is exactly satisfied:
$$\|\boldsymbol{h}(\gamma)\| = R.$$

To solve for $\lambda$ or $\gamma$ efficiently, we compute the eigenvalue decomposition of the symmetric matrix~$\boldsymbol{A}$:
$$\boldsymbol{A} = \boldsymbol{Q} \boldsymbol{\Lambda} \boldsymbol{Q}^\top, \quad \text{where } \boldsymbol{\Lambda} = \text{diag}(\lambda_1, \dots, \lambda_d).$$
This decomposition requires $O(d^\omega)$ running time. Letting $\tilde{\boldsymbol{b}} = \boldsymbol{Q}^\top \boldsymbol{b}$, the squared norm of the solution as a function of the multiplier $\alpha$ (representing either $\lambda$ or $\gamma$) becomes:
$$\phi(\alpha) = \|\boldsymbol{h}(\alpha)\|^2 = \sum_{i=1}^d \frac{\tilde{b}_i^2}{(\lambda_i + \alpha)^2}.$$
The convex function $\phi(\alpha)$ is strictly decreasing for $\alpha > -\lambda_1$. Consequently, the optimal multiplier can be found using the one-dimensional bisection search with $C$ iterations, and each iteration requires $\gO(d)$ running time. Thus, the overall running time is $\mathcal{O}(d^\omega + C d)$.
\end{proof}

\section{Numerical Experiments} \label{apx:exp}


In this section, we conduct numerical experiments on the regularized logistic regression problem of the  following form:
\begin{align*}
    \min_{\vx \in \sR^d} f(\vx) := \frac{1}{n} \sum_{i=1}^n \log (1 + \exp(- b_i \va_i^\top \vx)) +  \frac{1}{2n} \Vert \vx \Vert^2,
\end{align*}
where $\{(\va_i,b_i \}_{i=1}^n$ is a dataset with $n$ samples.
We set $\mH(\vx)$ as a mini-batch Hessian estimator with batch size $B \in \{100,1000,10000\}$, which provably satisfies Assumption~\ref{asm:delta-Hess} with high probability due to the matrix Hoeffding inequality \citep[Section 7]{agafonov2024inexact}.
We run experiments on the \texttt{covtype} dataset $(n = 581012, d=54)$ from LIBSVM \citep{chang2011libsvm} by comparing our proposed methods for $p \in \{2,3\}$ with baselines AGD \citep{nesterov1983method}, CRN \citep{nesterov2006cubic}, AITM \citep{agafonov2024inexact,agafonovadvancing}, and optimal tensor method (OTM) \citep{carmon2022optimal} for $p \in \{2,3\}$. The results are presented in Figure \ref{tab:fig-exp}, showcasing the superiority of our method.

\begin{table}[t]
    \centering
    \begin{tabular}{ccc}
    \includegraphics[scale=0.28]{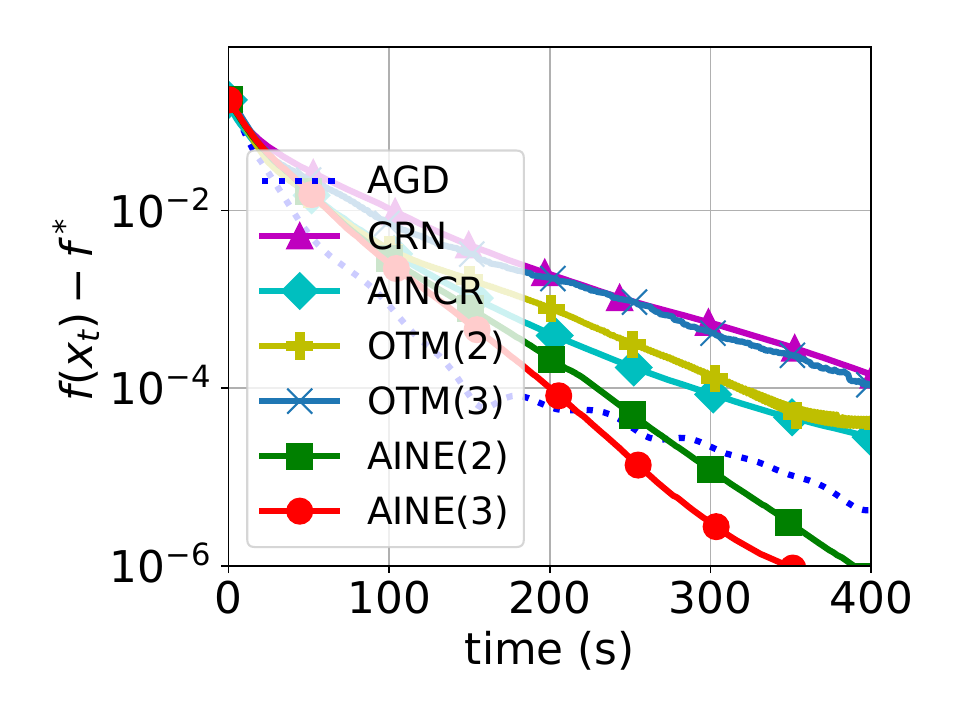}   & 
    \includegraphics[scale=0.28]{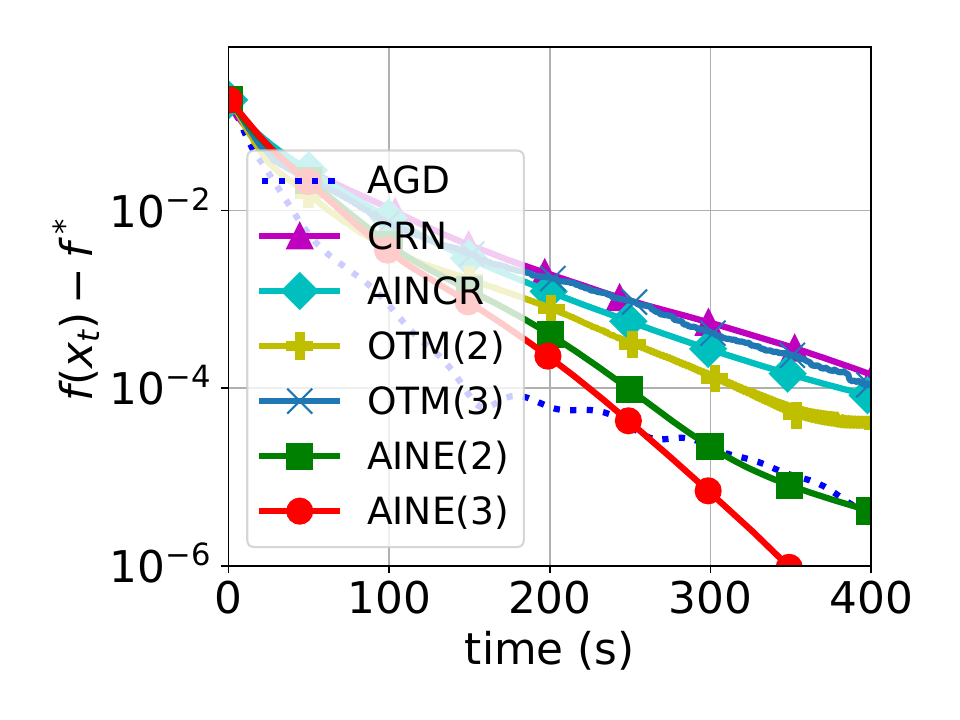} &
    \includegraphics[scale=0.28]{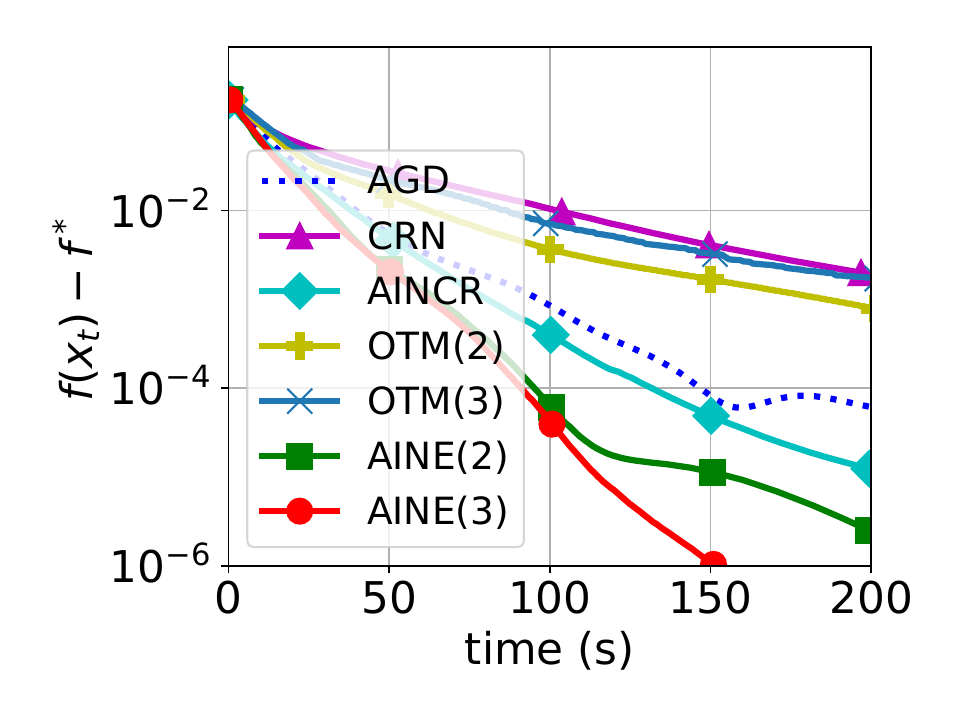}
    \\
    (a) $B= 100$     & (b) $B=1000$ & (c) $B=10000$
    \end{tabular}
    \caption{The results on logistic regression with different Hessian batch sizes $B$. We use AINE$(p)$ and OTM$(p)$ to denote the corresponding methods with order $p \in \{2,3\}$.}
    \label{tab:fig-exp}
\end{table}